\documentclass[11pt]{amsart}
\usepackage{amsmath,amssymb,amsthm,mathtools}
\usepackage{microtype}
\usepackage{xcolor}
\usepackage[colorlinks=true]{hyperref}
\numberwithin{equation}{section}
\mathtoolsset{showonlyrefs}
\allowdisplaybreaks[4]
\theoremstyle{plain}
\newtheorem{theorem}{Theorem}[section]
\newtheorem{lemma}[theorem]{Lemma}
\newtheorem{proposition}[theorem]{Proposition}

\theoremstyle{definition}
\newtheorem{definition}[theorem]{Definition}
\theoremstyle{remark}

\newcommand{\R}{\mathbb{R}}
\newcommand{\C}{\mathbb{C}}
\newcommand{\dd}{\,\mathrm{d}}
\newcommand{\Tr}{\operatorname{Tr}}
\newcommand{\EGL}{\mathcal E_{\mathrm{GL}}}
\newcommand{\Acal}{\mathcal A}
\newcommand{\Dcal}{\mathcal D}
\begin{document}

\title[Global Minimality of the Degree-One Vortex]
{On Brezis' Open Problem 2.2}

\author[H.G. Chen]{Hong-Ge Chen}
\address{Hong-Ge Chen, School of Mathematics and Statistics, Key Laboratory of Nonlinear Analysis \& Applications (Ministry of Education),
Central China Normal University, Wuhan 430079, China}
\email{hongge\_chen@whu.edu.cn}

\author[Y. Liu]{Yong Liu}
\address{Yong Liu, School of Mathematics and Statistics, Beijing Technology and Business University, Beijing, China}
\email{yliumath@btbu.edu.cn}

\author[J.C. Wei]{Juncheng Wei}
\address{Juncheng Wei, Department of Mathematics, Chinese University of Hong Kong, Shatin, New Territories, Hong Kong}
\email{wei@math.cuhk.edu.hk}

\author[W. Yang]{Wen Yang}
\address{Wen Yang, Department of Mathematics, Faculty of Science, University of Macau, Taipa, Macau, China.}
\email{wenyang@um.edu.mo}

\begin{abstract}
We prove that the global minimizer of the Ginzburg-Landau energy in the disk of radius $R$ with boundary value $ u(x)=\frac{x}{|x|}$ is the degree-one radial solution. This gives an affirmative answer to Open Problem~2.2 in Brezis' open-problem
list. The key idea is to compare the radial solution $f$ in the disk with the
degree-one radial solution $F$ in the whole plane.  Multiplying a
disk competitor by $F/f$ enables us to use the known minimality of the whole-plane
vortex without changing the boundary trace. The difference of the two
energies can be decomposed into Fourier modes.  Every nonzero mode is nonnegative, and the
zero mode is then handled by a Picone type identity.  
\end{abstract}

\maketitle

\section{Introduction}

The Ginzburg--Landau theory was introduced to describe superconductivity
\cite{deGennes,GinzburgLandau}.  In the full model, a complex-valued function
is coupled to a magnetic potential.  Abrikosov showed that this model also
describes quantized vortices in type-II superconductors \cite{Abrikosov}.  We
study the energy without the magnetic field.  Its Dirichlet term measures the spatial changes, while its potential term takes care of $|u|\neq1$.  The
books \cite{BBH,PacardRiviere} explain the problem without a magnetic field,
and \cite{SandierSerfaty} treats the magnetic model.

Vortices also occur in the problem without a magnetic field.  If the boundary
map has nonzero degree, every extension to the disk must vanish somewhere.
Around such a zero, the modulus is small and the phase makes a full turn.
Bethuel--Brezis--H\'elein and Struwe studied minimizers as
$\varepsilon\to0$.  They found the limiting vortex locations and the harmonic
map away from those points \cite{BBH,BBH93,Struwe}.  Sandier and Jerrard proved
lower bounds for the energy near the zeros \cite{Jerrard, SandierLower}.
Jerrard--Soner described the limiting vorticity by the distributional Jacobian
\cite{JerrardSoner}.  These results mainly concern the limit
$\varepsilon\to0$, where $\varepsilon$ is the scaling parameter associated to the problem.  The question considered here is for each fixed
$\varepsilon>0$.

The problem we are interested in is the following.  For $s>0$, let
$B_s(0):=\{x\in\R^2:|x|<s\}$, identify $\R^2$ with $\C$, and put
$g(x)=x$ on $\partial B_1(0)$.  We work in
\[
 H_g^1(B_1(0);\C)
 :=\bigl\{u\in H^1(B_1(0);\C):
 \Tr u=g\text{ in }H^{\frac{1}{2}}(\partial B_1(0);\C)\bigr\}.
\]
For $\varepsilon>0$, set
\[
 E_\varepsilon(u)
 :=\frac{1}{2}\int_{B_1(0)}|\nabla u|^2\dd x
 +\frac{1}{4\varepsilon^2}\int_{B_1(0)}(1-|u|^2)^2\dd x.
\]
Its critical points solve
\[
 \begin{cases}
 -\Delta u=\frac{1}{\varepsilon^2}u(1-|u|^2)&\text{in }B_1(0),\\[2mm]
 u(x)=x&\text{on }\partial B_1(0).
 \end{cases}
\]
The radial degree-one solution is
$U_\varepsilon(r,\theta)=f_\varepsilon(r)e^{i\theta}$, where
$f_\varepsilon$ is the positive solution of
\begin{equation*}
 \begin{cases}
 -f_\varepsilon''-\frac{1}{r}f_\varepsilon'
 +\frac{1}{r^2}f_\varepsilon
 =\frac{1}{\varepsilon^2}f_\varepsilon(1-f_\varepsilon^2)
 &\text{in }(0,1),\\[2mm]
 f_\varepsilon(0)=0,\qquad f_\varepsilon(1)=1.
 \end{cases}
\end{equation*}
Its existence, uniqueness, positivity, and monotonicity are
proved in \cite[Appendix~II]{BBH} and
\cite[Theorem~8.2]{PacardRiviere}.  Herv\'e--Herv\'e vary the initial slope
and show that one choice tends to one at infinity, while larger slopes reach
one at a finite radius \cite{HerveHerve}.  Chen--Elliott--Tang also use
shooting to prove uniqueness results for radial vortex solutions
\cite{ChenElliottTang}.  The precise ODE facts used below are collected in
Section~\ref{2-1}.

In principle, the radial boundary value does not force every vector-valued minimizer to be
radial.  Scalar rearrangement does not easily keep both the phase and the
boundary degree.  Lieb--Loss proved that the second variation at the radial
solution is nonnegative \cite{LiebLoss}.  Mironescu proved that it is strictly
positive and that the radial solution is a strict local minimizer for every
$\varepsilon>0$ \cite{MironescuStability}.  These results compare only maps
close to the radial solution and therefore do not rule out a distant map with
lower energy.

Some ranges of $\varepsilon$ were already understood.  Let $\lambda_1$ be the first
Dirichlet eigenvalue of $-\Delta$ in the disk.  If
$\varepsilon\geq\lambda_1^{-1/2}$, the energy is strictly convex, so the radial
solution is the unique critical point and minimizer
\cite[Section~2]{BrezisProblems}.  If $\varepsilon$ is sufficiently small,
Pacard--Rivi\`ere proved that the radial solution is the only critical point
\cite[Theorem~10.2]{PacardRiviere}.  Their proof uses estimates near the zeros
and weighted estimates for the linearized equation.  However, these two results do not
cover all intermediate values of $\varepsilon$, which is the most difficult case.

Higher-degree boundary values behave differently.  Comte--Mironescu used
bifurcation theory to construct nonradial branches when the boundary degree is
larger than one \cite{ComteMironescu}.  For degree one, Brezis stated the
all-parameter minimization question as Open Problem~6 in
\cite[Section~2]{BrezisSymmetry} and later as Open Problem~2.2 in
\cite{BrezisProblems}.  Open Problem~2.1 in the latter paper asks whether the
radial solution is the only solution of the equation.  Open Problem~2.2 asks
whether it minimizes the energy for every $\varepsilon>0$.  

Results on other domains help explain why the geometry matters.  On annuli,
Golovaty and Berlyand proved uniqueness of a radial minimizer on sufficiently
narrow annuli and showed that it can fail on wide annuli
\cite{GolovatyBerlyand}.  Berlyand and Mironescu related existence and vortex
formation on doubly connected domains to the capacity of the domain
\cite{BerlyandMironescu}.  Jimbo and Morita constructed stable rotational
solutions in a three-dimensional rotational domain and studied the linearized
equation \cite{JimboMorita}.  

There are also related results in higher dimensions.  The singular harmonic map
$x/|x|$ was studied in \cite{JagerKaul,BrezisCoronLieb,LinMap}.  These papers
show how its minimizing property depends on the dimension.  Millot--Pisante
described three-dimensional entire local minimizers under an energy-growth
bound, and Pisante proved local minimality of equivariant vortices in any
dimension \cite{MillotPisante,Pisante}.  For the unit ball, Ignat--Nguyen
proved local minimality of the radial vortex for every parameter when
$N\geq3$ \cite{IgnatNguyen}.  Ignat--Nguyen--Slastikov--Zarnescu proved global
minimality and uniqueness when $N\geq7$ \cite{INSZVortex}.  Their later paper
proves a more general uniqueness result under a one-sided condition on the
boundary value \cite{IgnatNguyenSlastikovZarnescuFull}.  Ignat--Rus proved
symmetry and global minimality for vortex sheets in cylinders \cite{IgnatRus}.
Ignat--Nahon--Nguyen proved minimality for gradient-field vortices in the unit
ball when $N\geq4$ \cite{IgnatNahonNguyen}.  None of these higher-dimensional
results applies to the two-dimensional disk considered here.

For our proof, the whole-plane degree-one vortex plays a key role.
Brezis--Merle--Rivi\`ere proved quantization results for entire solutions with
finite potential energy, and Shafrir studied their behavior at infinity
\cite{BrezisMerleRiviere,Shafrir}.  Sandier proved that every entire local
minimizer has finite potential energy \cite{SandierLocal}.  Together with
Mironescu's theorem, this shows that every nonconstant entire local minimizer
is the standard degree-one vortex, up to translation, rotation in the target,
and complex conjugation \cite{MironescuLocalMin}. 

A recent paper by del Pino--Liu--Musso--Wei--Yang studies the degree-two and
degree-three whole-plane vortices.  It proves their nondegeneracy, computes
their Morse indices, and gives explicit algebraic upper and lower bounds for
the corresponding radial functions
\cite[Proposition~2.3]{DPLMWY}.  It also uses Bernstein coefficients to check
polynomial signs; see \cite[Lemma~3.3]{DPLMWY}.  These two calculations are
close to parts of Section~\ref{4-1}.  That paper treats degrees two and three,
whereas the degree-one inequalities for $F$ used here are proved directly
from the degree-one ODE.

We can now state the main result.  It compares the radial solution with every
admissible map, not only with radial maps or maps close to it, and it assumes no
upper bound on the energy of the competitor. 

\begin{theorem}
\label{1-1}
For every $\varepsilon>0$ and every $u\in H_g^1(B_1(0);\C)$,
\[
 E_\varepsilon(u)\geq E_\varepsilon(U_\varepsilon).
\]
Moreover,
\[
 E_\varepsilon(u)=E_\varepsilon(U_\varepsilon)
 \quad\Longleftrightarrow\quad
 u=U_\varepsilon\quad
 %\text{almost everywhere} \
 \text{in}~B_1(0).
\]
\end{theorem}

This theorem resolves Brezis' open problem 2.2. We next outline the main steps of the proof.  First, it turns out to be more convenient to scale the unit disk to $B_R(0)$,
where $R=\varepsilon^{-1}$.  Let $f=f_R$ be the radial solution with
$f(R)=1$, and let $F$ be the radial solution in the whole plane that tends to
one at infinity.  Thus
\begin{align}
 f''+\frac{1}{r}f'-\frac{1}{r^2}f+(1-f^2)f&=0,
 &f(0)=0,\quad f(R)=1,\label{1-4}\\
 F''+\frac{1}{r}F'-\frac{1}{r^2}F+(1-F^2)F&=0,
 &F(0)=0,\quad\lim_{r\to\infty}F(r)=1.\label{1-5}
\end{align}
An integral identity for $f/F$ and a first-zero argument give
\[
 f(r)>F(r)\qquad(0<r\leq R).
\]

We then multiply a disk competitor $u$ by $F/f$.  Although $f$ and $F$ both
vanish at the origin, $F/f$ is smooth there.  Since $f(R)=1$, the new map has
the same boundary value as $Fe^{i\theta}$.  The minimality of the degree-one
vortex in the whole plane then gives
\[
 \EGL\left(\frac{F}{f}u;B_R(0)\right)
\geq \EGL(Fe^{i\theta};B_R(0)).
\]
Because $f$ and $F$ are different, a remainder still has to be controlled.

In the next step, we work first on a punctured disk, write $z=u/f$, and expand $z$ in
angular Fourier modes.  Subtracting the two radial ODE identities gives
\[
 \mathcal E_R(u)-\mathcal E_R(fe^{i\theta})
 =\left\{\EGL\left(\frac{F}{f}u;B_R(0)\right)
 -\EGL(Fe^{i\theta};B_R(0))\right\}+D_R[u].
\]
The term $D_R[u]$ is a nonnegative quartic integral plus one quadratic form for
each Fourier mode.  Every nonzero mode is nonnegative, so only the circular
average, namely the zero mode, needs a separate argument.

Finally, let $c_0$ be the circular average of $u$.  For a smooth competitor
with $u(0)\ne0$, the quotient $a_0=c_0/f$ behaves like $1/r$.  The change
$b=ra_0$ removes this singular behavior.  We apply the Picone identity
\cite{Picone} to the quadratic form for $b$, with a positive function made
from $f$ and $F$.  The remaining question is whether a scalar function
$S(r)$ is positive.  A first-zero argument reduces this to the sign of one
explicit polynomial.  We check that sign in the tensor Bernstein basis
\cite{Farouki}; all relevant coefficients are checked exactly.

There are two issues near the origin which should be dealt very carefully.  The identity containing $u/f$ is
first proved on an annulus.  As the inner radius tends to zero, two boundary
terms may have nonzero limits separately, but their sum tends to zero.
Moreover, a general $H^1$ map has no value at a single point.  We therefore
use a weighted one-dimensional space, prove density with a logarithmic cutoff,
and then pass to all Sobolev competitors.  This also gives the equality case. 

\section{Whole-plane comparison and the Fourier energy decomposition}
\label{2-1}
In this section, we analyze the exact comparison between an arbitrary disk map and
the radial solution.  After rescaling the disk, we compare its radial function
$f$ with the whole-plane function $F$ and show that $f/F$ is larger than one
and increasing.  The regular multiplier $F/f$ then gives a map with the same
boundary value as the whole-plane vortex, so its minimizing property can be
used.  For smooth competitors, subtracting the two radial energy identities
and expanding in angular Fourier modes gives a nonnegative quartic term and
nonnegative forms for every nonzero mode. The zero mode needs a further
argument.  We also check carefully that the two inner boundary terms created
by division by $f$ cancel at the origin.

Fix $\varepsilon>0$, set $R=\varepsilon^{-1}$, and for
$u\in H_g^1(B_1(0);\C)$ put $v(y)=u(\varepsilon y)$.  The change of
variables $x=\varepsilon y$ and the corresponding boundary dilation give
\[
 E_\varepsilon(u)
 =\mathcal E_R(v)
 :=\frac{1}{2}\int_{B_R(0)}|\nabla v|^2\dd y
 +\frac{1}{4}\int_{B_R(0)}(1-|v|^2)^2\dd y.
\]
This identifies $H_g^1(B_1(0);\C)$ with\[
 \Acal_R:=\bigl\{v\in H^1(B_R(0);\C):
 \Tr v=e^{i\theta}\text{ in }
 H^{\frac{1}{2}}(\partial B_R(0);\C)\bigr\}.\]
The same change of variables in the radial equation gives
$f_\varepsilon(s)=f_R\left(\frac{s}{\varepsilon}\right)$.

The functions $f$ and $F$ are the solutions of \eqref{1-4} and \eqref{1-5}.
For every $R>0$, \eqref{1-4} has a unique positive solution
$f=f_R\in C^\infty([0,R])$, while \eqref{1-5} has a unique positive
solution $F\in C^\infty([0,\infty))$.  On their respective open intervals,
$0<f,F<1$ and $f',F'>0$.  Both functions are analytic at $r=0$, with positive
initial slopes, and
\[
 F(r)=1-\frac{1}{2r^2}-\frac{9}{8r^4}+O(r^{-6}),\qquad
 F'(r)=\frac{1}{r^3}+\frac{9}{2r^5}+O(r^{-7})\quad(r\to\infty).
\]
The shooting theorem and its  remarks in
\cite[p.~428]{HerveHerve} state more precisely that there is a unique threshold
slope $\alpha$: its solution is $F$, whereas every slope larger than
$\alpha$ gives an increasing branch that reaches the level $1$ at a finite
first radius.  Supplementary Remark~2 states that, for each $R>0$, the
boundary problem has a unique solution and that it is the branch with some
slope $\beta>\alpha$. This also gives
the asserted uniqueness and monotonicity; see
\cite[Theorem~8.2]{PacardRiviere} as well.  The asymptotics can also be found in 
\cite[Lemma~A.1]{GravejatPacherieSmets}.

We will set $\alpha=F'(0)$ and $\beta=f'(0)$.  For $p\in\{F,f\}$ and
$a=p'(0)$, the degree-one analyticity result
\cite[Th\'eor\`eme, part~1, and (2), p.~428]{HerveHerve} gives
$p(r)=rA_p(r^2)$ with $A_p$ analytic, and substitution of
$p=ar+cr^3+er^5+O(r^7)$ into the radial equation tells us that
\begin{equation}\label{2-3}
 p(r)=ar-\frac{a}{8}r^3
 +\left(\frac{a^3}{24}+\frac{a}{192}\right)r^5
 +O(r^7).
\end{equation}

At this stage, we would like to introduce two important functions which will be used later:
\[
 k:=\frac{f}{F},
 \qquad
 y:=1-\frac{rF'}{F}.
\]

The two functions $f,F$ start with different slopes at the origin.  The disk
solution starts above the whole-plane solution.
Equivalently, their quotient $k=f/F$ stays above one and moreover it is increasing.  This is the content of the following 
\begin{proposition}
\label{2-4}
For every $0<r\leq R$, we have 
\[
 \beta>\alpha,
 \qquad k(r)>1,
 \qquad k'(r)>0.
\]
\end{proposition}

\begin{proof}
The preceding shooting classification gives $\beta>\alpha$.  Hence $k$
extends continuously to the origin with $k(0)=\frac{\beta}{\alpha}>1$.

Inserting $f=kF$ into the equation for $f$, followed by subtraction
of $k$ times the equation for $F$, tells us that
\[
 k''+\left(2\frac{F'}F+\frac1r\right)k'+F^2k(1-k^2)=0.
\]
Consequently,
\begin{equation}\label{2-5}
 (rF^2k')'=rF^4k(k^2-1).
\end{equation}
The expansion \eqref{2-3} also gives
\[
 k(r)=\frac\beta\alpha
 \left(1+\frac{\beta^2-\alpha^2}{24}r^4+O(r^6)\right),
\]
so $rF(r)^2k'(r)=O(r^6)$ at the origin.  Suppose that $k$ reaches $1$,
and let $r_*$ be its first contact with that level.  Since
$k(0)=\frac{\beta}{\alpha}>1$, one has $k>1$ on $(0,r_*)$.  Integrating
\eqref{2-5} from $0$ to $r<r_*$ yields
\[
 rF(r)^2k'(r)
 =\int_0^r sF(s)^4k(s)(k(s)^2-1)\dd s>0.
\]
Thus $k$ is strictly increasing on $(0,r_*)$, which is incompatible with
$k(r_*)=1<k(0)$.  Hence $k>1$ on $(0,R]$. This also implies that $k'>0$ throughout $(0,R]$. The proof is then completed.
\end{proof}

We shall also use the following logarithmic-slope consequences.  For
$0<r\leq R$,
\[
 y-\left(1-\frac{rf'}f\right)=\frac{rk'}k>0.
\]Hence 
\[ 0<1-\frac{rf'}f<y<1.\]
The first inequality can be proved as follows. For $p\in\{f,F\}$, the radial equation and the expansion at the origin give
\[
 r^2\left(p'(r)-\frac{p(r)}{r}\right)
 =-\int_0^r s^2(1-p(s)^2)p(s)\dd s<0,
\]
so $1-\frac{rp'}{p}>0$. 

The comparison map below uses both $f/F$ and its reciprocal.  At first these
quotients look undefined at the origin because both functions vanish there. However, their common linear factor removes this problem.  The next lemma gives the
regularity and the local expansion that we will use later.
\begin{lemma}
\label{2-6}
The functions $k=\frac{f}{F}$ and $\sigma=\frac{F}{f}$ extend smoothly to $r=0$, and
\begin{equation}
 k(r)
 =\frac{\beta}{\alpha}
 \left(1+\frac{\beta^2-\alpha^2}{24}r^4+O(r^6)\right),
 \label{2-7}
\end{equation}
In particular,
\[
 \sigma(|\cdot|)\in W^{1,\infty}(B_R(0);\R),
 \qquad
 0<\sigma\leq1.
\]
\end{lemma}

\begin{proof}
By the analytic formula preceding \eqref{2-3},
$f(r)=rA_f(r^2)$ and $F(r)=rA_F(r^2)$, where
$A_f(0)=\beta>0$ and $A_F(0)=\alpha>0$.  Hence
$k=\frac{A_f}{A_F}$ and $\sigma=\frac{A_F}{A_f}$ are analytic functions of $r^2$ near
the origin.  More explicitly, cancellation of the common quadratic term
in $\frac{A_f}{A_F}$ gives
\begin{align*}
 \frac{f(r)}{F(r)}
 &=\frac\beta\alpha
 \frac{1-\frac{r^2}{8}+\left(\frac{\beta^2}{24}+\frac{1}{192}\right)r^4+O(r^6)}
 {1-\frac{r^2}{8}+\left(\frac{\alpha^2}{24}+\frac{1}{192}\right)r^4+O(r^6)}\\
 &=\frac\beta\alpha
 \left(1+\frac{\beta^2-\alpha^2}{24}r^4+O(r^6)\right),
\end{align*}
which is \eqref{2-7}.  Taking the reciprocal, we obtain
\[
 \sigma(r)=\frac\alpha\beta
 \left(1-\frac{\beta^2-\alpha^2}{24}r^4+O(r^6)\right).
\]
Because $A_F(0)=\alpha>0$, the quotient $\frac{A_F}{A_f}$ is analytic in $r^2$
near the origin.  Its convergent local series is therefore a series in
$|x|^2=x_1^2+x_2^2$, so $x\mapsto\sigma(|x|)$ is smooth there.  Away from
the origin it is a smooth radial function, and compactness of
$\overline{B_R(0)}$ gives the $W^{1,\infty}$ bound.  Finally,
Proposition~\ref{2-4} gives $0<\sigma<1$ on $(0,R]$.
\end{proof}

\medskip 
We now use the ordering $f>F$ to compare an arbitrary disk map with the
standard vortex in the plane.  The comparison leaves a remainder, which will
be written exactly as a sum of angular Fourier terms.  Keeping this calculation
in the same section makes the role of $f$ and $F$ easier to follow.

For an open set $\mathcal O\subset\R^2$ and a map
$q\in H^1(\mathcal O;\C)$, write
\begin{equation*}
 \EGL(q;\mathcal O)
 :=\frac{1}{2}\int_{\mathcal O}|\nabla q|^2\dd x
 +\frac{1}{4}\int_{\mathcal O}(1-|q|^2)^2\dd x.
\end{equation*}
Let
\begin{equation*}
 V(x):=
 \begin{cases}
 F(|x|)\frac{x_1+ix_2}{|x|}=F(r)e^{i\theta},
 &x\in\R^2\setminus\{0\},\\[2mm]
 0,&x=0.
 \end{cases}
\end{equation*}

Since the degree one vortex on the plane is energy minimizing, we have
\begin{lemma}
\label{3-2}
For every $\rho>0$ and every $w\in H_0^1(B_\rho(0);\C)$,
\begin{equation*}
 \EGL(V+w;B_\rho(0))\geq\EGL(V;B_\rho(0)).
\end{equation*}
\end{lemma}

To use the preceding result, a disk competitor must first be put into the
whole-plane boundary class.  Multiplication by $F/f$ does exactly this because
$f(R)=1$.  Lemma~\ref{2-6} shows that the multiplier is also regular at the
origin.  We therefore obtain the whole-plane energy inequality inside $B_R$.
\begin{lemma}
\label{3-3}
If $u\in\Acal_R$ and $\sigma=\frac{F}{f}$, then
\[
 \sigma u-V\in H_0^1(B_R(0);\C),
 \qquad
 \EGL(\sigma u;B_R(0))\geq\EGL(V;B_R(0)).
\]
\end{lemma}

\begin{proof}
Lemma~\ref{2-6} and the Sobolev multiplier theorem give
$ \sigma u\in H^1(B_R(0);\C)$. The formula at the origin also gives
\[
 V(x)=A_F(|x|^2)(x_1+ix_2)\in H^1(B_R(0);\C).
\]
Since $f(R)=1$ and $\Tr u=e^{i\theta}$, their traces satisfy
\[
 \Tr(\sigma u)=F(R)e^{i\theta}=\Tr V.
\]
Thus $\sigma u-V\in H_0^1(B_R(0);\C)$.  The inequality follows from
Lemma~\ref{3-2} with $w=\sigma u-V$.
\end{proof}

We next need an exact formula for the energy of a map written as $pz$.
Division by $p$ is safe on an annulus because $p$ is strictly positive there.
The formula separates radial derivatives, angular derivatives, the quartic
term, and one boundary term.  We will later apply it to both $p=f$ and $p=F$
and subtract the two identities.

\begin{lemma}
\label{3-4}
Let $0<\rho<R$, let
\begin{equation*}
 A_{\rho,R}:=\{x\in\R^2:\rho<|x|<R\},
\end{equation*}
and let $p\in C^2([\rho,R];(0,\infty))$ satisfy
\begin{equation}\label{3-5}
 p''+\frac{1}{r}p'-\frac{1}{r^2}p+(1-p^2)p=0
 \quad\text{for }\rho<r<R.
\end{equation}
For every $z\in C^1(\overline{A_{\rho,R}};\C)$,
\begin{equation}
\begin{aligned}\label{3-6}
 &\EGL(pz;A_{\rho,R})-\EGL(pe^{i\theta};A_{\rho,R})\\
 &=\frac{1}{2}\int_{A_{\rho,R}}p^2
 \left(
 |z_r|^2+\frac{|z_\theta|^2-|z|^2}{r^2}
 \right)\dd x\\
 &\quad+\frac{1}{4}\int_{A_{\rho,R}}p^4(|z|^2-1)^2\dd x+\frac{1}{2}\int_0^{2\pi}
 \left[r p(r)p'(r)(|z(r,\theta)|^2-1)\right]_{r=\rho}^{r=R}
 \dd\theta.
\end{aligned}    
\end{equation}

\end{lemma}

\begin{proof}
Set $s=|z|^2-1$.  The radial and potential differences are
\begin{align*}
 |\partial_r(pz)|^2-p'^2&=p^2|z_r|^2+p'^2s+pp's_r,\\
 (1-p^2|z|^2)^2-(1-p^2)^2&=p^4s^2-2p^2(1-p^2)s.
\end{align*}
Together with
\[
 r^{-2}\bigl(|\partial_\theta(pz)|^2-|\partial_\theta(pe^{i\theta})|^2\bigr)
 =\frac{p^2}{r^2}(|z_\theta|^2-1),
\]
we obtain
\begin{align*}
 &\EGL(pz;A_{\rho,R})-\EGL(pe^{i\theta};A_{\rho,R})\\
 ={}&\frac{1}{2}\int_{A_{\rho,R}}p^2
 \left(|z_r|^2+\frac{|z_\theta|^2-1}{r^2}\right)\dd x
 +\frac{1}{4}\int_{A_{\rho,R}}p^4s^2\dd x\\
 &\quad+\frac{1}{2}\int_0^{2\pi}\!\int_\rho^R
 \{p'^2s+pp's_r-p^2(1-p^2)s\}r\dd r\dd\theta.
\end{align*}
For fixed $\theta$, integration by parts and
\eqref{3-5} yield
\begin{align*}
 &\int_\rho^R\{p'^2s+pp's_r-p^2(1-p^2)s\}r\dd r\\
 &=\left[rpp's\right]_{\rho}^{R}
 -\int_\rho^R rp\left(p''+\frac1r p'+(1-p^2)p\right)s\dd r\\
 &=\left[rpp's\right]_{\rho}^{R}-\int_\rho^R\frac{p^2}{r}s\dd r.
\end{align*}
Finally, $(|z_\theta|^2-1)-s=|z_\theta|^2-|z|^2$, and substitution proves
\eqref{3-6}.
\end{proof}

To proceed, let
\begin{equation*}
 \Acal_R^\infty
 :=\{u\in C^\infty(\overline{B_R(0)};\C):
 u(Re^{i\theta})=e^{i\theta}
 \text{ for every }\theta\in[0,2\pi]\}.
\end{equation*}
For $u\in\Acal_R^\infty$, define its Cartesian angular Fourier
coefficients by
\begin{equation*}
 c_m(r):=\frac{1}{2\pi}\int_0^{2\pi}
 u(r,\theta)e^{-im\theta}\dd\theta,
 \qquad m\in\mathbb Z,
 \qquad 0\leq r\leq R,
\end{equation*}
and, for $r>0$, set
\begin{equation*}
 a_m(r):=\frac{c_m(r)}{f(r)},
 \qquad
 b(r):=ra_0(r)=\frac{rc_0(r)}{f(r)}.
\end{equation*}
We will also write, for $0<r\leq R$,
\begin{equation*}
 d(r):=f(r)^2-F(r)^2>0,
 \qquad
 h(r):=\frac{d(r)}{r^2}>0.
\end{equation*}
The expansions of $f$ and $F$ at the origin give $d(0)=0$ and
the removable smooth extension
\[
 h(0):=\beta^2-\alpha^2>0.
\]

  Every nonzero Fourier mode of $u/f$ has enough Cartesian vanishing to give a finite
integral.  The zero mode may have a simple pole when $u(0)\ne0$, so we replace
$a_0$ by the bounded variable $b=ra_0$.  The next lemma shows that the new
boundary term cancels exactly with the boundary term from the annular formula.

\begin{lemma}
\label{3-7}
Let $u\in\Acal_R^\infty$.  For $m\neq0$ the integral\[
 Q_m[a_m]
 :=\pi\int_0^R
 \left(r d|a_m'|^2+(m^2-1)\frac{d}{r}|a_m|^2\right)\dd r\]
is finite and nonnegative.  Moreover, $b(R)=0$ and
\[
 Q_0[b]
 :=\pi\int_0^R\left(r h|b'|^2+h'|b|^2\right)\dd r,
\]
is absolutely convergent.  The inner boundary term from
Lemma~\ref{3-4} and the zero-mode integration by
parts cancel:
\begin{equation}\label{3-9}
 \lim_{\rho\to0^{+}}
 \left(
 \mathfrak B_\rho+\pi h(\rho)|b(\rho)|^2
 \right)=0,
\end{equation}
where
\begin{equation}\label{3-10}
 \mathfrak B_\rho
 :=-\frac{1}{2}\int_0^{2\pi}
 \rho\bigl(f(\rho)f'(\rho)-F(\rho)F'(\rho)\bigr)
 \left(\left|\frac{u(\rho,\theta)}{f(\rho)}\right|^2-1\right)
 \dd\theta.
\end{equation}
\end{lemma}

\begin{proof}
Put $z=\frac{u}{f}=\sum_{m\in\mathbb Z}a_m(r)e^{im\theta}$ on $0<r\leq R$.
For each $r>0$, Parseval gives
\begin{align*}
 \int_0^{2\pi}|z_r|^2\dd\theta&=2\pi\sum_m|a_m'|^2,\\
 \int_0^{2\pi}|z_\theta|^2\dd\theta&=2\pi\sum_m m^2|a_m|^2,\\
 \int_0^{2\pi}|z|^2\dd\theta&=2\pi\sum_m|a_m|^2.
\end{align*}
After multiplication by the radial weights and integration, these identities
give
\begin{align}
 &\frac{1}{2}\int_{A_{\rho,R}}d
 \left(|z_r|^2+\frac{|z_\theta|^2-|z|^2}{r^2}\right)\dd x\notag\\
 &=\pi\sum_{m\in\mathbb Z}\int_\rho^R
 \left(r d|a_m'|^2+(m^2-1)\frac{d}{r}|a_m|^2\right)\dd r.
 \label{3-11}
\end{align}
Fix $m\ne0$ and choose $N\geq|m|+1$.  With
$\zeta=x_1+ix_2$, the Cartesian Taylor expansion at the origin has the form
\[
 u(re^{i\theta})=\sum_{j+\ell\leq N}\gamma_{j\ell}
 r^{j+\ell}e^{i(j-\ell)\theta}+O(r^{N+1}).
\]
The remainder may be differentiated once with an $O(r^N)$ bound.  Taking
the $m$th angular coefficient retains only pairs with $j-\ell=m$; since
$j+\ell\geq|j-\ell|=|m|$, we obtain
\[
 c_m=O(r^{|m|}),\qquad c_m'=O(r^{|m|-1}).
\]
Since $f=\beta r+O(r^3)$ and
$d=(\beta^2-\alpha^2)r^2+O(r^4)$, it follows that
\[
 a_m=O(r^{|m|-1}),\qquad a_m'=O(r^{|m|-2}),
\]
and hence
$rd|a_m'|^2+\frac{d}{r}|a_m|^2=O(r^{2|m|-1})$.  Thus $Q_m[a_m]$ is finite and,
because $d>0$ and $m^2-1\geq0$, nonnegative.

The boundary condition gives $c_0(R)=b(R)=0$.  Since $a_0=\frac{b}{r}$ and
$d=r^2h$,
\[
 r d\left(|a_0'|^2-\frac{|a_0|^2}{r^2}\right)
 =r h|b'|^2-h(|b|^2)'.
\]
Thus integration by parts leads to
\begin{equation}
 \pi\int_\rho^R rd\left(|a_0'|^2-\frac{|a_0|^2}{r^2}\right)\dd r
 =\pi\int_\rho^R
 \left(r h|b'|^2+h'|b|^2\right)\dd r
 +\pi h(\rho)|b(\rho)|^2.
 \label{3-12}
\end{equation}

To justify both the bulk limit and its boundary correction, angular
averaging of the Taylor expansion at the origin gives
\[
 c_0(r)=u(0)+O(r^2),\qquad c_0'(r)=O(r).
\]
Writing $f(r)=rA_f(r^2)$, we have
\[
 b(r)=\frac{c_0(r)}{A_f(r^2)}=\frac{u(0)}\beta+O(r^2),
 \qquad b'(r)=O(r),
\]
while $h=\beta^2-\alpha^2+O(r^2)$ and $h'=O(r)$.  Hence
$rh|b'|^2=O(r^3)$ and $h'|b|^2=O(r)$, proving absolute convergence of
$Q_0[b]$.

We now compute the two boundary limits.  Angular averaging removes the
linear part of the Cartesian Taylor expansion, so
\[
 \frac1{2\pi}\int_0^{2\pi}|u(r,\theta)|^2\dd\theta
 =|u(0)|^2+O(r^2).
\]
The expansions of $f$ and $F$ also give
\[
 r(ff'-FF')=(\beta^2-\alpha^2)r^2+O(r^4),\qquad
 f^2=\beta^2r^2+O(r^4).
\]
Consequently,
\[
 \frac1{2\pi}\int_0^{2\pi}
 \left(\frac{|u(r,\theta)|^2}{f(r)^2}-1\right)\dd\theta
 =\frac{|u(0)|^2}{\beta^2r^2}+O(1),
\]
and substitution into \eqref{3-10} yields
\begin{align*}
 \mathfrak B_\rho
 &=-\pi\left(1-\frac{\alpha^2}{\beta^2}\right)|u(0)|^2+o(1),\\
 \pi h(\rho)|b(\rho)|^2
 &=\pi\left(1-\frac{\alpha^2}{\beta^2}\right)|u(0)|^2+o(1).
\end{align*}
The second line follows directly from
$h(\rho)=\beta^2-\alpha^2+O(\rho^2)$ and
$b(\rho)=\frac{u(0)}{\beta}+O(\rho^2)$.
Their sum tends to zero, proving
\eqref{3-9}.
In particular, neither boundary term may be discarded before the two
radial identities are subtracted: each has a nonzero limit whenever
$u(0)\ne0$, and only their signed sum is removable.
\end{proof}

We can now subtract the two annular identities and let the inner radius tend
to zero.  The whole-plane comparison gives one part of the energy difference.
The part left over is a sum of Fourier forms and a quartic integral.  The next
proposition states this exact formula for smooth competitors. We define $D_R[u]$ through
$$ 
 \mathcal E_R(u)-\mathcal E_R(fe^{i\theta})=\EGL(\sigma u;B_R(0))-\EGL(V;B_R(0))+D_R[u].$$
\begin{proposition}
\label{3-13}
For $u\in\Acal_R^\infty$,   we have
\begin{align*}
 D_R[u]
 ={}&Q_0[b]+\sum_{m\neq0}Q_m[a_m]\\
 &+\frac{1}{4}\int_{B_R(0)}
 \left(1-\left(\frac{F}{f}\right)^4\right)
 (|u|^2-f^2)^2\dd x.
\end{align*}
\end{proposition}

\begin{proof}
Fix $0<\rho<R$ and set $z=\frac{u}{f}$ on $A_{\rho,R}$. Let $\sigma=\frac{F}{f}$. Then $u=fz$,
$\sigma u=Fz$, and $z=e^{i\theta}$ at $r=R$, so the outer boundary terms
in Lemma~\ref{3-4} vanish.  Applying that lemma with
$p=f$ and $p=F$ and subtracting gives
\begin{align*}
 &\bigl[\EGL(u;A_{\rho,R})-\EGL(fe^{i\theta};A_{\rho,R})\bigr]
 -\bigl[\EGL(\sigma u;A_{\rho,R})-\EGL(V;A_{\rho,R})\bigr]\\
 ={}&\frac12\int_{A_{\rho,R}}d
 \left(|z_r|^2+\frac{|z_\theta|^2-|z|^2}{r^2}\right)\dd x
 +\mathfrak B_\rho\\
 &\quad+\frac14\int_{A_{\rho,R}}
 \left(1-\left(\frac Ff\right)^4\right)(|u|^2-f^2)^2\dd x.
\end{align*}
Here we used
\[
 (f^4-F^4)(|z|^2-1)^2
 =\left(1-\left(\frac Ff\right)^4\right)(|u|^2-f^2)^2.
\]

By \eqref{3-11}, the quadratic term is the sum of
its angular modes.  For the zero mode,
\eqref{3-12} and
\eqref{3-9} give
\begin{align*}
 &\lim_{\rho\downarrow0}\biggl\{\mathfrak B_\rho
 +\pi\int_\rho^R rd
 \left(|a_0'|^2-\frac{|a_0|^2}{r^2}\right)\dd r\biggr\}\\
 &=\lim_{\rho\downarrow0}\biggl\{\mathfrak B_\rho
 +\pi h(\rho)|b(\rho)|^2
 +\pi\int_\rho^R(rh|b'|^2+h'|b|^2)\dd r\biggr\}=Q_0[b].
\end{align*}
For $m\ne0$ the mode integrands are nonnegative.  Tonelli's theorem and
monotone convergence therefore yield
\[
 \lim_{\rho\downarrow0}\pi\sum_{m\ne0}\int_\rho^R
 \left(rd|a_m'|^2+(m^2-1)\frac d r|a_m|^2\right)\dd r
 =\sum_{m\ne0}Q_m[a_m].
\]
Since $0<\frac{F}{f}<1$, the quartic integrand is nonnegative.  It is integrable
because $\frac{F}{f}$ is bounded and $u,f$ are smooth on the closed disk; monotone
convergence therefore sends its annular integral to the integral over
$B_R(0)$.

It remains to identify the limit of the left-hand side.  The formulas at the
origin give
\[
 fe^{i\theta}=A_f(|x|^2)(x_1+ix_2),\qquad
 V=A_F(|x|^2)(x_1+ix_2).
\]
Moreover, $\sigma(r)=\frac{A_F(r^2)}{A_f(r^2)}$, so $\sigma u$ is smooth at the
origin.  Thus the energy densities of $u$, $fe^{i\theta}$, $\sigma u$,
and $V$ are bounded near $0$, and for each of these maps $q$,
\[
 \EGL(q;B_\rho(0))=O(\rho^2)\qquad(\rho\downarrow0).
\]
Subtracting the four small-ball energies shows that the left-hand side of
the annular identity converges to $D_R[u]$.  Combining this fact with the
mode limits proves the asserted decomposition.
\end{proof}
 
The preceding formula is an exact nonlinear identity.  Its terms show three different effects of
replacing the disk solution $f$ by the smaller whole-plane solution $F$.
The coefficient $d=f^2-F^2$ measures the loss in the weighted Dirichlet
part, whereas $1-\left(\frac{F}{f}\right)^4$ measures the corresponding loss in the quartic
part.  Thus the quartic remainder is nonnegative without any smallness
assumption on $u-fe^{i\theta}$.  After angular decomposition, the same
Dirichlet loss becomes
\[
 \pi\int_0^R\left(rd|a_m'|^2+(m^2-1)\frac d r|a_m|^2\right)\dd r.
\]
For $|m|\geq2$ both terms are nonnegative.  For $m=\pm1$ the
zeroth-order coefficient vanishes, so $Q_{\pm1}$ measures only radial
variation; equality later forces $a_{\pm1}$ to be constant, and the
boundary trace determines those constants.  The mode $m=0$ is the only one
different: its angular coefficient has the wrong sign, and positivity is
recovered only after $a_0=\frac{b}{r}$ converts it into $Q_0[b]$.  Hence the
zero-mode estimate below is the only part not supplied directly by $F<f$.

The annular identity is true at any fixed inner radius $\rho$. At that stage division by $f$ is
regular and Lemma~\ref{3-4} is an ordinary
integration-by-parts formula.  The limit $\rho\downarrow0$ is then taken
in three different ways.  For $m\ne0$, the Cartesian vanishing order gives
an integrable nonnegative density, so Tonelli and monotone convergence
apply.  For $m=0$, neither a separate boundary limit nor monotone
convergence is available: one must first combine $\mathfrak B_\rho$ with
the boundary term in \eqref{3-12}.  Finally,
the quartic term is treated only after it has been rewritten as
$\left(1-\left(\frac{F}{f}\right)^4\right)(|u|^2-f^2)^2$.  

The zeroth angular
coefficient of a smooth Cartesian map need not vanish at the origin, while
division by $f$ creates precisely a simple pole.  The
renormalization $b=ra_0$ removes that pole, and
$h=\frac{f^2-F^2}{r^2}$ has a positive finite limit.  If $u(0)\ne0$, the two
boundary terms have the separate nonzero limits displayed above, but their
sum vanishes.  Thus the apparent contribution at the origin is created by
division by $f$, which vanishes there.  Once these two boundary terms are combined, the
nonzero-mode and quartic terms pass to the disk monotonically, while the
four annular energies converge by absolute continuity.

\section{Picone identity and positivity of the zero Fourier mode}
\label{4-1}
This section supplies the sign needed for the zero Fourier mode left open by
the preceding decomposition.  We first prove two estimates for the
whole-plane function $F$: its initial slope is larger than $1/2$, and its
logarithmic slope satisfies $y\leq F^2$.  We then choose a positive Picone
multiplier that rewrites the zero-mode form as a square plus a term weighted by
a scalar function $S$.  A closed first-order system reduces the positivity of
$S$ to a first-zero calculation, and the remaining polynomial sign is checked
by elementary estimates and Bernstein coefficients.  This gives the
positive weight needed for the zero-mode estimate.

We begin with following
\begin{lemma}
\label{4-2}
The whole-plane initial slope satisfies
\begin{equation*}
 \alpha>\frac{1}{2}.
\end{equation*}
\end{lemma}

\begin{proof} 
Let $G(r)=\frac{r}{\sqrt{r^2+4}}$ and write $F=GZ$.  A direct calculation gives
\[
 G''+\frac1rG'-\frac1{r^2}G+G(1-G^2)
 =\frac{3r^3}{(r^2+4)^{\frac{5}{2}}}>0.
\]
Writing $F=GZ$ and collecting the terms containing $Z'$ gives
\[
 0=GZ''+\left(2G'+\frac Gr\right)Z'
 +\frac{3r^3Z}{(r^2+4)^{\frac{5}{2}}}+ZG^3(1-Z^2).
\]
Multiplication by $rG$ turns the first two terms into $(rG^2Z')'$ and
leads to
\begin{equation}\label{4-3}
 (rG^2Z')'
 +\frac{3r^4GZ}{(r^2+4)^{\frac{5}{2}}}
 +rG^4Z(1-Z^2)=0.
\end{equation}
At a local minimum $r_0$ with $0<Z(r_0)<1$, one has
\[
 (rG^2Z')'(r_0)=r_0G(r_0)^2Z''(r_0)\geq0,
\]
whereas the remaining terms in \eqref{4-3} are
strictly positive.  Such a local minimum is therefore impossible.

Note that
$G(r)=\frac{r}{2}\left(1-\frac{r^2}{8}+\frac{3r^4}{128}+O(r^6)\right).$  Thus
\begin{align*}
 Z(r)&=2\alpha\left(1+\frac{16\alpha^2-7}{384}r^4+O(r^6)\right),\\
 Z'(r)&=\frac{\alpha(16\alpha^2-7)}{48}r^3+O(r^5).
\end{align*}
Suppose to the contrary that $\alpha\leq\frac{1}{2}$.  Then $Z(0)\leq 1$, $16\alpha^2-7<0$, so $Z'<0$ and
$0<Z<1$ on $(0,r_*]$ for some $r_*>0$.  

Now let us choose $0<r_1<r_2\leq r_*$.  Since
$Z(r_2)<Z(r_1)<1$ and $Z(r)\to1$, there is $L>r_2$ with
$Z(L)>Z(r_2)$.  The minimum of $Z$ on $[r_1,L]$ is therefore attained at
an interior point and has value in $(0,1)$, contradicting the preceding
observation.  Hence $\alpha>\frac{1}{2}$.
\end{proof}
We remark that this lower bound can also be deduced from the lower bound of $F$ obtained in \cite{DPLMWY}.

The next estimate connects the size of $F$ with its logarithmic slope.
The quantity $y=1-rF'/F$ lies between zero and one, but that alone is not
enough for the later contact calculation purpose.  We will need the sharper bound
$y\leq F^2$. We have not been able to find a reference for this inequality, so we give detailed proof here.
\begin{lemma}
\label{4-4}
For every $r>0$,
\begin{equation}\label{4-5}
 F(r)^2\geq y(r).
\end{equation}
\end{lemma}

\begin{proof}
Put $a=F^2$, $A=a-y$, and $J=r^2(1-a)-a(4-3a)$.  Since
$\frac{rF'}{F}=1-y$, the radial equation gives
\begin{equation}\label{4-6}
 ra'=2a(1-y),\qquad ry'=y^2-2y+r^2(1-a),
\end{equation}
Indeed, if $q=\frac{rF'}{F}=1-y$, then the radial equation yields
\[
 rq'=q+\frac{r^2F''}{F}-q^2=1-r^2(1-a)-q^2,
\]
which is the second identity; the first follows by differentiating
$a=F^2$.  Consequently, at every zero of $A$,
\begin{equation}\label{4-7}
 rA'=a(4-3a)-r^2(1-a)=-J.
\end{equation}
Differentiating $J$ and using $ra'=2a(1-y)$ yields
\[
 rJ'=2r^2(1-a)-(r^2+4-6a)\,2a(1-y).
\]
At a zero of $J$, where $r^2=\frac{a(4-3a)}{1-a}$, this becomes
\begin{equation}\label{4-8}
 rJ'=2a(4-3a)-2a(1-y)\frac{3a^2-6a+4}{1-a}.
\end{equation}
If also $A\leq0$, then $a\leq y<1$, and hence $0<1-y\leq1-a$.
Because $3a^2-6a+4>0$, \eqref{4-8} implies
\begin{equation}\label{4-9}
 rJ'\geq2a(4-3a)-2a(3a^2-6a+4)=6a^2(1-a)>0.
\end{equation}

Near the origin, \eqref{2-3} gives
\[
 a=\alpha^2r^2-\frac{\alpha^2}{4}r^4+O(r^6),\qquad
 y=\frac14r^2+\left(\frac1{96}-\frac{\alpha^2}{6}\right)r^4+O(r^6).
\]
At infinity, the asymptotic formulas for $F$ give
\[
 a=1-r^{-2}-2r^{-4}+O(r^{-6}),\qquad
 y=1-r^{-2}-5r^{-4}+O(r^{-6}).
\]
Together with Lemma~\ref{4-2}, these yield
\begin{align*}
 A(r)&=\left(\alpha^2-\frac14\right)r^2+O(r^4)>0 &&(r\downarrow0),\\
 A(r)&=3r^{-4}+O(r^{-6})>0 &&(r\to\infty).
\end{align*}
In particular, every component of $\{A<0\}$ is relatively compact in
$(0,\infty)$.  If $(c,d)$ is such a component, then
$A(c)=A(d)=0$, and \eqref{4-7} gives
$J(c)\geq0\geq J(d)$.  We use the consequence of
\eqref{4-9} that $J$ cannot pass from a positive value to a
nonpositive one while $A\leq0$: the first zero after a positive point
would have $J'\leq0$, contrary to that equation.  If $J(c)=0$, the same
inequality gives $J'(c)>0$, so $J>0$ immediately to the right of $c$.
It follows that $J>0$ throughout $(c,d)$.  Thus $J(d)=0$; but
\eqref{4-9} at $d$ gives $J'(d)>0$, which forces $J(r)<0$ for
$r<d$ sufficiently close to $d$, again a contradiction.  Hence
$A\geq0$.
\end{proof} 

\bigskip
The two estimates above give the signs needed in the only Fourier mode
that is not already nonnegative.  We now build a positive comparison function
for this mode and use it in a Picone identity.  The rest of this section proves
the one-variable inequality required by that identity.

The angular gap $m^2-1\geq0$ already controls every nonzero mode.  Only
$m=0$ has a negative angular term, and it leaves the form
\[
 Q_0[b]=\pi\int_0^R(rh|b'|^2+h'|b|^2)\dd r.
\]
A direct estimate of $h'$ does not work because $h'$ has no fixed sign.
We instead choose a positive function $\phi$ and use the standard Picone
identity.  The formulas for $f/F$ make all second derivatives cancel and
leave one scalar expression $S$.

Let us introduce\[
 M:=k^2-1>0,
 \qquad
 \eta:=\frac{rk'}{k^2-1}=\frac{rk'}{M},
 \qquad
 t:=rF.\]
 Note that \begin{equation}\label{5-9}
 \eta(r)=\frac{\alpha\beta}{6}r^4+O(r^6)
 \qquad\text{as }r\to0^{+}.
\end{equation}
These variables contain the information about $f$ and $F$ used below. More precisely, $t$ describes $F$, while $M,\eta$ describe the differences between $f$ and $F$.  Their
derivatives can all be written in terms of the other variables, with no
second derivatives of $f$ or $F$.  This makes it possible to differentiate
$S$ at a possible first zero.  The next lemma gives the required system and
the initial size of $\eta$.
\begin{lemma}
\label{5-3}
On $(0,R]$,
\begin{align}
 r\frac{\dd t}{\dd r}&=t(2-y),\label{5-4}\\
 r\frac{\dd y}{\dd r}&=y^2-2y+r^2-t^2,
 \label{5-5}\\
 r\frac{\dd k}{\dd r}&=M\eta,\label{5-6}\\
 r\frac{\dd M}{\dd r}&=2kM\eta,\label{5-7}\\
 r\frac{\dd\eta}{\dd r}
 &=kt^2-2(1-y)\eta-2k\eta^2.
 \label{5-8}
\end{align}
\end{lemma}

\begin{proof}
From $t=rF$ and $\frac{rF'}{F}=1-y$, one obtains $rt'=t(2-y)$, while the second
identity is \eqref{4-6}.  The definition of $\eta$ gives
$rk'=M\eta$ and hence $rM'=2krk'=2kM\eta$.

It remains to derive the equation for $\eta$.  Since
$h=\frac{f^2-F^2}{r^2}=\frac{F^2M}{r^2}$, the quotient equation
\eqref{2-5} and $rk'=M\eta$ give
\begin{equation}\label{5-10}
 (F^2M\eta)'=(rF^2k')'=rF^4kM.
\end{equation}
Moreover, we have
\[
 r\frac{(F^2M)'}{F^2M}=2\frac{rF'}F+\frac{rM'}M=2(1-y)+2k\eta.\]
Expanding the left-hand side of \eqref{5-10}, dividing
by $F^2M$, and multiplying by $r$, we obtain
\[
 r\eta'+\{2(1-y)+2k\eta\}\eta=r^2F^2k=kt^2,
\]
which is \eqref{5-8}. This finishes the proof.
\end{proof}

The variable $\eta$ is positive because $k$ is increasing.  We also need an
upper bound that is uniform in the disk radius $R$.  The line $r/\sqrt2$ is a
strict upper comparison for the equation in Lemma~\ref{5-3}.  This bound will be used
only when checking the sign of one coefficient at a zero of $S$.
\begin{lemma}
\label{5-12}
For every $0<r\leq R$,
\begin{equation*}
 0<\eta(r)<\frac{r}{\sqrt{2}}.
\end{equation*}
\end{lemma}

\begin{proof}
Proposition~\ref{2-4} gives $k'>0$ and $M>0$, hence
$\eta>0$.  Set $g(r):=\eta(r)-\frac{r}{\sqrt2}$.  By
\eqref{5-9}, there is $\delta>0$ such that $g<0$ on
$(0,\delta]$.  If $g$ were nonnegative somewhere, continuity would give a
least zero $r_0\in(\delta,R]$.  Then $g<0$ on $(0,r_0)$ and $g(r_0)=0$,
so $D_-g(r_0)\geq0$.  Since $g$ is $C^1$ from the left also when $r_0=R$,
this is $r_0\eta'(r_0)\geq\frac{r_0}{\sqrt2}$.  At the contact value,
however, \eqref{5-8} gives
\[
 r_0\eta'(r_0)=-k(r_0)r_0^2\{1-F(r_0)^2\}
 -\sqrt2r_0\{1-y(r_0)\}<0,
\]
because $k>1$ by Proposition~\ref{2-4}, $0<F<1$, and $0<y<1$ by the
logarithmic-slope bounds.  This contradiction proves the
strict upper bound.
\end{proof}

The multiplier can be found directly from logarithmic derivatives.  Since
\[
 \frac{rh'}h=2(k\eta-y),
\]
we look for a function whose logarithmic derivative is $-(y+k\eta)$.
For a product $F^a r^bM^c$ one has
\[
 r\frac{\dd}{\dd r}\log(F^a r^bM^c)
 =(a+b)-ay+2c k\eta.
\]
Requiring this to equal $-(y+k\eta)$ gives $a=1$, $b=-1$, and
$c=-\frac{1}{2}$, which gives the function below.

To proceed, let us define the positive Picone multiplier
\begin{equation}\label{5-13}
 \phi:=\frac{F^2}{r\sqrt{f^2-F^2}}
 =\frac{F}{r\sqrt{M}}>0
 \quad\text{on }(0,R],
\end{equation}
The Picone identity will contain the expression
$-(rh\phi')'+h'\phi$.  The main point is that this expression has a simple
sign after division by the positive factor $rh\phi$.  All second derivatives
can be removed with the first-order system above.  The result is the scalar
function $S$ in the next lemma.
\begin{lemma}
\label{5-14}
The function $\phi$ satisfies
\begin{equation}\label{5-15}
 \frac{-(r h\phi')'+h'\phi}{r h\phi}=\frac{S}{r^2},
\end{equation}
where
\begin{equation}\label{5-16}
 S:=r^2(1-F^2+f^2)-4y-2y^2-\eta^2
 =r^2+Mt^2-4y-2y^2-\eta^2.
\end{equation}
\end{lemma}

\begin{proof}
Since $h=\frac{F^2M}{r^2}$ and $\phi=\frac{F}{r\sqrt M}$, put
\[
 p:=\frac{rh'}h=2(k\eta-y),
 \qquad q:=\frac{r\phi'}\phi=-(y+k\eta).
\]
These identities follow from $\frac{rF'}{F}=1-y$ and $\frac{rM'}{M}=2k\eta$.  A
logarithmic differentiation gives
\begin{equation}\label{5-17}
 \frac{-(r h\phi')'+h'\phi}{r h\phi}
 =\frac1{r^2}\{p-rq'-pq-q^2\}.
\end{equation}
Write temporarily $\zeta=k\eta$.  From \eqref{5-6} and
\eqref{5-8},
\begin{equation}\label{5-18}
 r\zeta'=k^2t^2-2\zeta(1-y)-(M+2)\eta^2.
\end{equation}
Since $p=2(\zeta-y)$ and $q=-(y+\zeta)$, the numerator in
\eqref{5-17} is
\[
 p-rq'-pq-q^2=ry'+r\zeta'+2\zeta-2y+\zeta^2-3y^2-2y\zeta.
\]
Substituting \eqref{5-5} and
\eqref{5-18}, and using $k^2=M+1$ and
$\zeta^2=k^2\eta^2$, gives
\begin{align*}
 p-rq'-pq-q^2
 &=r^2+(k^2-1)t^2-4y-2y^2-\eta^2\\
 &=r^2+Mt^2-4y-2y^2-\eta^2=S.
\end{align*}
Equation \eqref{5-15} follows, and
$r^2+Mt^2=r^2(1-F^2+f^2)$ gives the other form of $S$.
\end{proof}

The key
feature of the residual identity is that all second derivatives of $f$ and $F$ have disappeared.
The variables $M,y,\eta,t$ satisfy the closed first-order system above,
so at a zero of $S$ its derivative is determined solely by the value of
these variables at that point.  This makes a first-contact argument
possible and explains the otherwise special-looking multiplier $\phi$.
The choice of $\phi$ also fits the bounded variable $b$.
Near the origin both $h$ and $\phi$ approach positive constants, so the
Picone identity neither imposes an extra condition $b(0)=0$ nor
loses the possible value $\frac{u(0)}{\beta}$.  Once the identity is closed on the
weighted space, pointwise positivity of $S$ makes the second Picone weight
strictly positive and rules out a nonzero null vector.

Notice also that \eqref{5-15} holds throughout
$(0,R]$, not only at a contact point.  Thus $S$ is first differentiated
using its second expression, and the relation $S=0$ is imposed only
afterward to eliminate $r^2$ from the differentiated identity.

We now insert the last identity into the quadratic form for $b$.  Completing
the square gives a nonnegative derivative term and a term containing $S$.
On an annulus there is also a boundary contribution, so it must be kept until
the inner radius tends to zero.  The following is the exact Picone formula we
will use.
\begin{lemma}
\label{5-19}
Let $0<\rho<R$, and let
$b\in C^1([\rho,R];\C)$ satisfy $b(R)=0$.  Then
\begin{equation}
\begin{aligned}
 &\int_\rho^R\left(r h|b'|^2+h'|b|^2\right)\dd r\\
 ={}&\int_\rho^R r h\phi^2
 \left|\left(\frac{b}{\phi}\right)'\right|^2\dd r
 +\int_\rho^R\frac{hS}{r}|b|^2\dd r+\left[r h\frac{\phi'}{\phi}|b|^2\right]_{r=\rho}^{r=R}.
 \label{5-20}
\end{aligned}    
\end{equation}

\end{lemma}

\begin{proof}
Because $\phi$ is real and positive, we may take the real part of the cross
term.  In particular,
\[
 (|b|^2)'=2\Re(b'\overline b),
\]
and hence
\[
 \phi^2\left|\left(\frac b\phi\right)'\right|^2
 =|b'|^2-\frac{\phi'}\phi(|b|^2)'
 +\left(\frac{\phi'}\phi\right)^2|b|^2.
\]
Using
$\left(\frac{rh\phi'}{\phi}\right)'=\frac{(rh\phi')'}{\phi}-rh\left(\frac{\phi'}{\phi}\right)^2$,
we obtain the pointwise identity
\[
 rh|b'|^2+h'|b|^2
 =rh\phi^2\left|\left(\frac b\phi\right)'\right|^2
 +\frac{-(rh\phi')'+h'\phi}{\phi}|b|^2
 +\left(rh\frac{\phi'}\phi|b|^2\right)'.
\]
Integration over $[\rho,R]$ and
$\frac{-(rh\phi')'+h'\phi}{\phi}=\frac{hS}{r}$ prove \eqref{5-20}.  The
hypothesis $b(R)=0$ makes the contribution at $R$ vanish; the bracket is
retained for the later limit $\rho\downarrow0$.
\end{proof}

The expansion at the origin gives the initial sign of $S$, but does not show
that $S$ is monotone.  At a possible first zero, the equation $S=0$ removes
$r^2$ from the derivative and the bound $F^2\geq y$ gives a lower bound for
$t^2$.  Before using that lower bound, we will check its coefficient with the
estimate $\eta<r/\sqrt2$.  We also set $U=rk'/(ky)$, which always lies in
$(0,1)$ and keeps the parameter range independent of $R$.  The next lemma
gives the derivative at a zero of $S$ and rewrites $\eta$ in terms of
$(M,U,y)$.

\begin{lemma}
\label{5-21}
Define $U:=\frac{rk'}{ky}$.  Then $M>0$, $0<U<1$, $0<y<1$, and
\begin{equation}\label{5-22}
 \eta=\frac{\sqrt{1+M}}{M}Uy.
\end{equation}
At every $r\in(0,R]$ such that $S(r)=0$,
\begin{align}
 \frac r2\frac{\dd S}{\dd r}
 ={}&2k\eta^3+(1-4y)\eta^2\notag\\
 &+t^2\left(k\eta(M-1)+(M+2)y+3M+2\right)
 -6y^3-8y^2.
 \label{5-23}
\end{align}
\end{lemma}

\begin{proof}
The logarithmic-slope bounds above and
Proposition~\ref{2-4} give the stated ranges.
Furthermore,
$Uy=\frac{rk'}{k}=\frac{M\eta}{k}$, and $k=\sqrt{1+M}$ gives
\eqref{5-22}.  Differentiating
$S=r^2+Mt^2-4y-2y^2-\eta^2$ and multiplying by $\frac{r}{2}$ gives
\[
 \frac r2S'=r^2+\frac12(rM')t^2+Mt(rt')-2(1+y)ry'-\eta(r\eta').
\]
Using \eqref{5-4}--\eqref{5-8} first yields
\begin{align*}
 \frac r2S'
 ={}&r^2+kM\eta t^2+Mt^2(2-y)
 -2(1+y)(y^2-2y+r^2-t^2)\\
 &-\eta\{kt^2-2(1-y)\eta-2k\eta^2\}\\
 ={}&-(1+2y)r^2
 +t^2\{kM\eta+2M-My+2+2y-k\eta\}\\
 &-2y^3+2y^2+4y+2(1-y)\eta^2+2k\eta^3.
\end{align*}
At a zero of $S$, $r^2=-Mt^2+4y+2y^2+\eta^2$.  We then find that 
\[
 \frac r2S'=2k\eta^3+(1-4y)\eta^2-6y^3-8y^2
 +t^2\{k\eta(M-1)+(M+2)y+3M+2\},
\]
which is the desired identity \eqref{5-23}.
\end{proof}

At a contact point, define
\begin{equation*}
 C_t:=k\eta(M-1)+(M+2)y+3M+2.
\end{equation*}
Note that the derivative formula of $S$ contains $C_t t^2$.  We must know the sign of $C_t$
before replacing $t^2$ by a lower bound.  The cases $M\geq1$ and $0<M<1$
are slightly different, but both give a strict positive sign.  The same
calculation also gives the lower bound for $t^2$ needed in the next step.
\begin{lemma}
\label{5-24}
At every zero of $S$, $ C_t>0$, 
and
\begin{equation*}
 t^2\geq
 \frac{y\bigl(\eta^2+2y(y+2)\bigr)}{1+My}.
\end{equation*}
\end{lemma}

\begin{proof}
At $S=0$, equation \eqref{5-16} gives
\[
 r^2(1+MF^2)=\eta^2+2y(y+2),
 \qquad
 t^2
 =\frac{F^2}{1+MF^2}\{\eta^2+2y(y+2)\}.
\]
Since $F^2\geq y$ by Lemma~\ref{4-4} and
$q\mapsto\frac{q}{1+Mq}$ has derivative $(1+Mq)^{-2}>0$, all denominators
are positive, and therefore
\[
 \frac{F^2}{1+MF^2}\geq\frac{y}{1+My}.
\]
This proves the asserted lower bound without changing the sign of any
factor.

If $M\geq1$, then $C_t>0$ is immediate.  If $0<M<1$, the first identity
in \eqref{5-16} writes the middle fraction below as $r^2$; hence
Lemma~\ref{5-12} and $F^2\geq y$ give
\[
 2\eta^2<\frac{\eta^2+2y(y+2)}{1+MF^2}
 \leq\frac{\eta^2+2y(y+2)}{1+My}.
\]
Multiplication by $1+My>0$ yields
$\eta^2(1+2My)<2y(y+2)$.  Since $k^2=1+M$ and $k,\eta>0$,
\[
 k\eta<\left(\frac{2(1+M)y(y+2)}{1+2My}\right)^{\frac{1}{2}}.
\]
The quantity under the square root is strictly smaller than
$4(y+1)^2$, because
\[
 4(y+1)^2(1+2My)-2(1+M)y(y+2)
 =2y^2+4y+4+M(8y^3+14y^2+4y)>0.
\]
Thus $k\eta<2(y+1)$.  Since $1-M>0$, multiplication by $1-M$
preserves strictness, and
\begin{align*}
 C_t
 &=(M+2)y+3M+2-(1-M)k\eta\\
 &>(M+2)y+3M+2-2(1-M)(y+1)=M(3y+5)>0.
\end{align*}
\end{proof}

We need to prove that certain polynomial is positive in some region. For this purpose, we will use the Bernstein
basis because nonnegative Bernstein coefficients give positivity inside the
region.   We
state the formulas here.
The same test and conversion formula are used in
\cite[Lemma~3.3 and (3.18)]{DPLMWY}.
\begin{lemma}
\label{5-26}
Let $B_i^n(s)=\binom ni s^i(1-s)^{n-i}$.  If
$P(s)=\sum_{j=0}^np_js^j$, then
\[
 P(s)=\sum_{i=0}^n
 \left(\sum_{j=0}^ip_j\frac{\binom ij}{\binom nj}\right)B_i^n(s).
\] 
\end{lemma}

After the substitutions above, the derivative at a zero of $S$ is bounded
below by one explicit polynomial.  Its variables satisfy $M>0$ and
$0<U,Y<1$.  The following proposition proves positivity on this whole set.  It also states exactly how the polynomial
enters the derivative of $S$.
\begin{proposition}
\label{5-27}
For $M>0$ and $0<U,Y<1$, define
\begin{equation*}
 N(M,U,Y):=\sum_{j=0}^4A_j(U,Y)M^j,
\end{equation*}
where
\begin{align*}
 A_0={}&U^3Y(2-Y),\\
 A_1={}&U^2(UY^2+4UY+2Y^2-2Y+1),\\
 A_2={}&U(5U^2Y^2+2U^2Y-UY^2+2UY+U-2Y^2-4Y),
\\
 A_3={}&Y(3U^3Y-3U^2Y+4U^2+4Y+6),\\
 A_4={}&2(UY^2+2UY-2Y^2+Y+6).
\end{align*}
Then 
\begin{equation}\label{5-28}
 N(M,U,Y)>0.
\end{equation}
As a consequence, at  zero of $S$ where $r>0$, 
\begin{equation}\label{5-29}
 \frac r2\frac{\dd S}{\dd r}
 \geq\frac{y^2}{M^3(1+My)}N(M,U,y)>0.
\end{equation}
\end{proposition}

\begin{proof}
At a zero of $S$, put $Y=y(r)$.  Lemmas~\ref{5-21} and
\ref{5-24}, together with \eqref{5-22}, give a lower bound for
$rS'/2$.  The substitutions needed below are
\[
 k\eta=\frac{(1+M)UY}{M},\qquad
 \eta^2=\frac{(1+M)U^2Y^2}{M^2},\qquad
 k\eta^3=\frac{(1+M)^2U^3Y^3}{M^3},
\]
and
\[
 C_t=\frac{(1+M)UY(M-1)}{M}+(M+2)Y+3M+2.
\]
Lemma~\ref{5-24} also shows that $C_t>0$.  We may therefore insert
the contact lower bound without reversing the inequality:
\[
 \frac r2S'\geq\mathcal R
 :=2k\eta^3+(1-4Y)\eta^2-6Y^3-8Y^2
 +C_t\frac{Y\{\eta^2+2Y(Y+2)\}}{1+MY}.
\]
Multiplying by the positive factor $M^3(1+MY)/Y^2$ gives
\begin{align*}
 \frac{M^3(1+MY)}{Y^2}\mathcal R
 ={}&2(1+M)^2U^3Y(1+MY)\\
 &+(1-4Y)(1+M)U^2M(1+MY)\\
 &+C_t\{M(1+M)U^2Y+2M^3(Y+2)\}\\
 &-(6Y+8)M^3(1+MY).
\end{align*}
Direct expansion of the right-hand side gives
\[
 \sum_{j=0}^4A_j(U,Y)M^j=N(M,U,Y).
\]
This proves the first inequality in \eqref{5-29}.

It remains to check that $N>0$.  Four coefficients already have a clear
sign:
\begin{align*}
 A_0&>0,\\
 \frac{A_1}{U^2}
 &=UY^2+4UY+2\left(Y-\frac12\right)^2+\frac12>0,\\
 \frac{A_3}{Y}
 &=3U^2Y(U-1)+4U^2+4Y+6>0,\\
 \frac{A_4}{2}
 &=UY^2+2UY-2Y^2+Y+6>4.
\end{align*}
Only $A_2$ may be negative.  We treat it in two regions.

Suppose first that $M\geq U$.  Since $A_3>0$, we have
\[
 N\geq A_0+MA_1+M^2(A_2+UA_3)+M^4A_4.
\]
Write $\mathbf B_n(s)=(B_0^n(s),\ldots,B_n^n(s))^{\mathsf T}$.
The conversion formula in Lemma~\ref{5-26} gives
\[
 \frac{A_2+UA_3}{U}
 =\mathbf B_3(U)^{\mathsf T}
 \begin{pmatrix}
 0&1&4\\
 \frac13&\frac53&\frac{14}3\\
 \frac23&\frac{10}3&8\\
 1&6&17
 \end{pmatrix}
 \mathbf B_2(Y)>0.
\]
Thus $N>0$ in this region.

Now suppose that $0<M<U$ and set $Z=M/U\in(0,1)$.  After removing
the positive factor $U^3$, the polynomial to be checked is
\begin{align*}
 \frac{N(UZ,U,Y)}{U^3}
 ={}&Y(2-Y)\\
 &+Z\{UY^2+4UY+2Y^2-2Y+1\}\\
 &+Z^2\{5U^2Y^2+2U^2Y-UY^2+2UY+U-2Y^2-4Y\}\\
 &+Z^3Y\{3U^3Y-3U^2Y+4U^2+4Y+6\}\\
 &+2UZ^4\{UY^2+2UY-2Y^2+Y+6\}.
\end{align*}
Its tridegree in $(Z,U,Y)$ is at most $(4,3,2)$.  Applying
Lemma~\ref{5-26} successively in these three variables gives $60$
tensor Bernstein coefficients. More precisely, the polynomial in the right hand side can be written as
\[\sum_{i=0}^{4}\sum_{j=0}^{3}\sum_{\ell=0}^{2}b_{ij\ell}B_i^4(Z)B_j^3(U)B_\ell^2(Y),\]
where \[B_i^n(s)=\binom{n}{i}s^i(1-s)^{n-i}.\]The total number of these Bernstein coefficients is \[(4+1)(3+1)(2+1)=5\cdot4\cdot3=60.\] Exactly four are zero, the other $56$ are positive,  and the smallest positive coefficient is $1/4$.
Consequently it is positive on $(0,1)^3$, and hence
$N>0$ also when $M<U$.
\end{proof}

We now combine the local expansion of $S$ near origin with the previous contact calculation.  The
expansion tells us $S$ positive near the origin.  If a first zero existed, Lemma~\ref{5-24} would allow the lower
bound for $t^2$ to be used with the correct sign.  The variables at that point
would lie in the range of Proposition~\ref{5-27}, which makes the 
derivative of $S$ positive.  This contradicts the behavior required at a
first zero. This is the content of the following

\begin{proposition}
\label{5-32}
For every $0<r\leq R$,
\begin{equation}\label{5-33}
 S(r)>0.
\end{equation}
\end{proposition}

\begin{proof}  Applied to $F$,
\eqref{2-3} gives, for $r$ small, 
\[
 F^2=\alpha^2r^2+O(r^4),\qquad
 y=\frac14r^2+\left(\frac1{96}-\frac{\alpha^2}{6}\right)r^4+O(r^6).
\]
The same expansion for $f$ yields
\[
 r^2(1-F^2+f^2)=r^2+(\beta^2-\alpha^2)r^4+O(r^6).
\]
Moreover, $2y^2=\frac{r^4}{8}+O(r^6)$, while
\eqref{5-9} gives $\eta^2=O(r^8)$.  Hence
\begin{align}
 S(r)
 &=r^2+(\beta^2-\alpha^2)r^4
 -\left\{r^2+\left(\frac1{24}-\frac{2\alpha^2}{3}\right)r^4\right\}
 -\frac18r^4+O(r^6)\notag\\
 &=\left(\beta^2-\frac{\alpha^2}{3}-\frac16\right)r^4+O(r^6).
 \label{5-34}
\end{align}
By Proposition~\ref{2-4} and
Lemma~\ref{4-2},
\[
 \beta^2-\frac{\alpha^2}{3}-\frac16
 >\frac{2\alpha^2}{3}-\frac16=\frac{4\alpha^2-1}{6}>0.
\]
Hence $S>0$ on $(0,\delta]$ for some $\delta>0$.
If $S$ had a zero in $(\delta,R]$, let $r_0$ be its least zero.
Continuity then tells us that $S>0$ on $(0,r_0)$, so $S'(r_0)\leq0$.   At this point
$M(r_0)>0$ and $0<U(r_0),y(r_0)<1$, by
Proposition~\ref{2-4} and the logarithmic-slope bounds.
Thus every denominator in the contact estimate is strictly positive and
Proposition~\ref{5-27} gives
\[
 \frac{r_0}{2}S'(r_0)
 \geq\frac{y(r_0)^2N(M(r_0),U(r_0),y(r_0))}
 {M(r_0)^3(1+M(r_0)y(r_0))}>0,
\]
a contradiction. This completes the proof.
\end{proof}

\section{Extension of the energy identity to Sobolev maps}
\label{6-1}

This section extends the defect identity analyzed in the previous sections from smooth competitors to the full
Sobolev class.  The main issue is the zero Fourier mode near the origin, where
division by $f$ may be singular and a general Sobolev map has no point value.
We place the bounded variable $b$ in a weighted one-dimensional space, prove
density by a logarithmic cutoff, and use it to extend the Picone identity and
identify its equality case.  Sobolev slicing and angular Fourier projections
then control the remaining modes and preserve the boundary trace.  Finally, an
approximation with fixed boundary data gives the exact decomposition
$D_R[u]\geq0$ for every $u\in\Acal_R$, which is used in the last section. 

We start with the following

\begin{definition}
\label{6-2}
Let $\Dcal_R$ be the space
\begin{align}
 \Dcal_R:=\biggl\{b\in H_{\mathrm{loc}}^1((0,R);\C):{}&
 \int_0^R r\bigl(|b'(r)|^2+|b(r)|^2\bigr)\dd r<\infty,\notag\\
 &\lim_{r\to R^{-}}b(r)=0\biggr\},
 \label{6-3}
\end{align}
with norm
\begin{equation*}
 \|b\|_{\Dcal_R}^2
 :=\int_0^R r\bigl(|b'|^2+|b|^2\bigr)\dd r.
\end{equation*}
The value at $R$ is the one-dimensional trace from $H^1((\delta,R))$. 
\end{definition}

\begin{proof}[Completeness of $\Dcal_R$]
A $\Dcal_R$-Cauchy sequence is Cauchy in the two copies of
$L^2((0,R),r\dd r)$ and, for every $\delta>0$, in
$H^1((\delta,R))$.  The local limits identify the weighted derivative,
and the one-dimensional trace at $R$ passes to the limit.  Thus the limit
belongs to $\Dcal_R$ and the convergence is in its norm.
\end{proof}

To include $r=0$, we need uniform bounds for the three weights
in the Picone formula.  Their local expansions follow from the radial
series already proved.  Positivity then extends the bounds to the closed
interval.  The next lemma collects exactly the estimates used in the
closure argument.
\begin{lemma}
\label{6-4}
As $r\to0^{+}$,
\begin{align}
 h(r)&=\beta^2-\alpha^2+O(r^2),
 &h'(r)&=O(r),
 \label{6-5}\\
 \phi(r)&=\frac{\alpha^2}{\sqrt{\beta^2-\alpha^2}}+O(r^2),
 &\phi'(r)&=O(r),
 \label{6-6}  
\end{align}
There are constants $c_R,C_R>0$ such that, for every $0\leq r\leq R$,
\begin{equation}\label{6-8}
 c_R\leq h(r)\leq C_R,
 \qquad
 c_R\leq\phi(r)\leq C_R,
 \qquad
 |h'(r)|+|\phi'(r)|\leq C_Rr.
\end{equation}
Moreover,
\begin{equation*}
 r h(r)\frac{\phi'(r)}{\phi(r)}=O(r^2)
 \qquad\text{as }r\to0^{+},
\end{equation*}
and
\begin{equation}\label{6-9}
 0<\frac{h(r)S(r)}{r}\leq C_Rr
 \qquad\text{for }0<r\leq R.
\end{equation}
\end{lemma}

\begin{proof}
The expansions of $f$ and $F$ at the origin and Proposition~\ref{2-4} give
\[
 h=\beta^2-\alpha^2+O(r^2),
 \qquad
 \phi=\frac{\alpha^2}{\sqrt{\beta^2-\alpha^2}}+O(r^2).
\]
Indeed, $h$ and $\phi$ are analytic functions of $r^2$ near $0$, whence
$h',\phi'=O(r)$.   
Positivity and compactness on $[0,R]$ now give
\eqref{6-8} and
$\frac{rh\phi'}{\phi}=O(r^2)$.  Finally,
Proposition~\ref{5-32} gives $\frac{hS}{r}>0$, while the origin estimate
is $O(r^3)$; compactness away from $0$ gives
\eqref{6-9}.
\end{proof}

We next show that test functions supported away from both $0$ and $R$ are
enough.  Near $r=0$, an ordinary linear cutoff has too much weighted
derivative cost.  A logarithmic cutoff makes that cost tend to zero.  A value
truncation is used first so that the cutoff error is controlled.
\begin{lemma}
\label{6-10}
The space $C_c^\infty((0,R);\C)$ is dense in $\Dcal_R$ with respect to
$\|\cdot\|_{\Dcal_R}$.
\end{lemma}

\begin{proof}
A point can be removed at arbitrarily small weighted cost in two dimensions.
A linear cutoff near $0$ has gradient cost of order one and does not give the
needed approximation, whereas the logarithmic cutoff below has cost
$|\log\delta|^{-1}$.  Because that cost is multiplied by $|b|^2$, we first
truncate the values of $b$ and only then remove the puncture.

For $L>0$, define the complex-valued amplitude truncation
\[
 T_L(\xi)=\begin{cases}\xi,&|\xi|\leq L,\\ \frac{L\xi}{|\xi|},&|\xi|>L.
 \end{cases}
\]
It is Lipschitz, $|(T_L(b))'|\leq|b'|$ almost everywhere, and it preserves
the trace at $R$.  Thus $T_L(b)\in\Dcal_R$.  For integer $L=n$, local
absolute continuity shows that $T_n(b)\to b$ and
$(T_n(b))'\to b'$ almost everywhere, while
\[
 |T_n(b)-b|^2\leq4|b|^2,\qquad
 |(T_n(b))'-b'|^2\leq4|b'|^2.
\]
Dominated convergence with respect to $r\dd r$ gives
$T_n(b)\to b$ in $\Dcal_R$.  It is therefore enough to approximate a
bounded $b$, say $|b|\leq L$.  For
$0<\delta<\min\{1,R\}$, let
\[
 \chi_\delta(r)=
 \begin{cases}
 0,&0<r\leq\delta^2,\\
 \frac{\log\left(\frac{r}{\delta^2}\right)}{|\log\delta|},&\delta^2<r<\delta,\\
 1,&\delta\leq r<R.
 \end{cases}
\]
On $(\delta^2,\delta)$ one has
$\chi_\delta'(r)=\frac{1}{r|\log\delta|}$, and therefore
\[
 \int_0^Rr|\chi_\delta'|^2\dd r=|\log\delta|^{-1}.
\]
The three relevant errors satisfy
\begin{align*}
 \int_0^Rr|(1-\chi_\delta)b|^2\dd r
 &\leq\tfrac12L^2\delta^2,\\
 \int_0^Rr|(1-\chi_\delta)b'|^2\dd r
 &\leq\int_0^\delta r|b'|^2\dd r,\\
 \int_0^Rr|\chi_\delta'b|^2\dd r
 &\leq L^2|\log\delta|^{-1}.
\end{align*}
All tend to zero as $\delta\downarrow0$, and
$(\chi_\delta b-b)'=(\chi_\delta-1)b'+\chi_\delta'b$; hence
$\chi_\delta b\to b$ in $\Dcal_R$.  For fixed $\delta$, the restriction
of $\chi_\delta b$ to $(\delta^2,R)$ lies in
$H_0^1((\delta^2,R))$.  On this interval the weighted and ordinary
$H^1$ norms are equivalent, so standard density followed by extension by
zero gives approximants in $C_c^\infty((0,R);\C)$.  A diagonal choice
completes the proof.
\end{proof}

Lemma~\ref{5-19} gives the Picone identity only for functions away from the
origin.  The estimates in Lemma~\ref{6-4} make both sides continuous in the
weighted norm.  Density therefore removes the temporary restriction away
from $r=0$.
The strict positivity of the $S$-term also identifies the only equality
case.
\begin{proposition}
\label{6-11}
For every $b\in\Dcal_R$, the form
\begin{equation*}
 Q_0[b]
 :=\pi\int_0^R\left(r h|b'|^2+h'|b|^2\right)\dd r
\end{equation*}
is well defined and satisfies
\begin{align}
 \frac{Q_0[b]}{\pi}
 ={}&\int_0^R r h\phi^2
 \left|\left(\frac{b}{\phi}\right)'\right|^2\dd r
 +\int_0^R\frac{hS}{r}|b|^2\dd r.
 \label{6-12}
\end{align}
In particular,
\begin{equation*}
 Q_0[b]\geq0,
 \qquad
 Q_0[b]=0\quad\Longleftrightarrow\quad b=0
 \quad\text{in }\Dcal_R.
\end{equation*}
\end{proposition}

\begin{proof}
By Lemma~\ref{6-4}, both sides of
\eqref{6-12} define continuous quadratic forms on
$\Dcal_R$: indeed, $h,\phi,\phi^{-1}$ are bounded,
$|h'|+|\phi'|\leq C_Rr$, and $0<\frac{hS}{r}\leq C_Rr$; in particular,
$h'|b|^2$ is absolutely integrable for $b\in\Dcal_R$.  The identity holds on
$C_c^\infty((0,R);\C)$ by Lemma~\ref{5-19}, because
the boundary term vanishes, and hence on $\Dcal_R$ by
Lemma~\ref{6-10}.  No extra condition $b(0)=0$ is
being imposed: the space $\Dcal_R$ has no trace or prescribed value at the
puncture, and in particular it includes functions having a nonzero finite
limit at $r=0$.  Both integrands on the right are
nonnegative; since $\frac{hS}{r}>0$ on $(0,R]$, equality holds only for $b=0$.
\end{proof}

We now return from the one-dimensional zero mode to maps on the disk.  Sobolev
slicing gives a periodic angular function for almost every radius.  Parseval's
identity then separates the radial and angular parts of the $H^1$ norm.  We
also need the Fourier projections to preserve the boundary trace.
\begin{lemma}
\label{6-13}
Let $u\in H^1(B_R(0);\C)$.  For almost every $r\in(0,R)$, the angular
slice $\theta\mapsto u(re^{i\theta})$ belongs to
$H^1((0,2\pi);\C)$ and has the same trace at $0$ and $2\pi$.  Define

\begin{equation}\label{6-14}
 c_m(r):=\frac{1}{2\pi}\int_0^{2\pi}
 u(r,\theta)e^{-im\theta}\dd\theta,
 \qquad m\in\mathbb Z.
\end{equation}
Then $c_m\in H_{\mathrm{loc}}^1((0,R);\C)$ for every $m\in\mathbb Z$, and
\begin{equation}\label{6-15}
 c_m'(r)=\frac{1}{2\pi}\int_0^{2\pi}
 u_r(r,\theta)e^{-im\theta}\dd\theta
 \quad\text{for almost every }r\in(0,R).
\end{equation}
For $\zeta\in L^2((0,2\pi);\C)$, define its $m$-th angular projection by
\[
 (P_m\zeta)(\theta):=\left(\frac{1}{2\pi}\int_0^{2\pi}
 \zeta(\varphi)e^{-im\varphi}\dd\varphi\right)e^{im\theta}.\]
For functions on the disk, we use the same notation for
$P_mu(r,\theta):=c_m(r)e^{im\theta}$.
Then
\begin{align}
 \int_{B_R(0)}|u|^2\dd x
 &=2\pi\sum_{m\in\mathbb Z}\int_0^R r|c_m|^2\dd r,
 \label{6-17}\\
 \int_{B_R(0)}|u_r|^2\dd x
 &=2\pi\sum_{m\in\mathbb Z}\int_0^R r|c_m'|^2\dd r,
\\
 \int_{B_R(0)}\frac{|u_\theta|^2}{r^2}\dd x
 &=2\pi\sum_{m\in\mathbb Z}\int_0^R
 \frac{m^2}{r}|c_m|^2\dd r.
 \label{6-18}
\end{align}
The angular Fourier projection commutes with the trace map\[
 \Tr:H^1(B_R(0);\C)\to H^{\frac{1}{2}}(\partial B_R(0);\C).\]
\end{lemma}

\begin{proof}
On each annulus the assertions follow from Sobolev slicing and Parseval on
the periodic cylinder.  Multiplication by the polar weights $r,r,r^{-1}$,
followed by $\delta\downarrow0$, gives
\eqref{6-17}--\eqref{6-18}; smooth density
extends the identities to $H^1(B_R(0))$.  Finally, $P_m$ is bounded on
both $H^1(B_R(0))$ and $H^{\frac{1}{2}}(\partial B_R(0))$, so the identity
$\Tr(P_mu)=P_m(\Tr u)$ follows from the smooth case.  In particular,
$c_m(R)e^{im\theta}=P_m(\Tr u)$.
\end{proof}
  The next lemma gives both the nonzero-mode estimate and the
weighted estimate for the zero mode.
\begin{lemma}
\label{6-20}
Let $u\in H^1(B_R(0);\C)$ satisfy
\begin{equation*}
 P_0(\Tr u)=0
 \quad\text{in }H^{\frac{1}{2}}(\partial B_R(0);\C).
\end{equation*}
Let $c_m$ be defined by \eqref{6-14}, and set, for
$m\neq0$,
\begin{equation}\label{6-21}
 a_m:=\frac{c_m}{f}
 \quad\text{on }(0,R).
\end{equation}
Then
\begin{equation*}
 0\leq\sum_{m\neq0}Q_m[a_m]
 \leq C_R\|u\|_{H^1(B_R(0);\C)}^2.
\end{equation*}
For the zero coefficient,
\begin{equation}\label{6-22}
 b:=\frac{rc_0}{f}
\end{equation}
belongs to $\Dcal_R$, and
\begin{equation*}
 \|b\|_{\Dcal_R}
 \leq C_R\|u\|_{H^1(B_R(0);\C)}.
\end{equation*}
Both assignments are continuous under strong $H^1$ convergence: if
$u_j\to u$ and all their boundary traces have zero angular mean, then
\[
 \sum_{m\neq0}Q_m[a_{m,j}]\to\sum_{m\neq0}Q_m[a_m],
 \qquad Q_0[b_j]\to Q_0[b].
\]
\end{lemma}

\begin{proof}
Write $f=rq$, where $q$ is positive and smooth on $[0,R]$; hence
$q,q^{-1}\in W^{1,\infty}(0,R)$ and $\left|\frac{f'}{f}\right|\leq\frac{C_R}{r}$.  Direct
differentiation of $a_m=\frac{c_m}{f}$, together with
$d=f^2\left(1-\left(\frac{F}{f}\right)^2\right)$, gives for $m\ne0$
\[
 rd|a_m'|^2+(m^2-1)\frac d r|a_m|^2
 \leq C_R\left(r|c_m'|^2+\frac{m^2}{r}|c_m|^2\right).
\]
Summation and Lemma~\ref{6-13} prove the nonzero-mode bound.
Angular averaging is contractive in $H^1$.  Together with
\[
 b=\frac{c_0}{q},\qquad
 b'=\frac{c_0'}q-\frac{c_0q'}{q^2}.
\]
these two identities, this gives the zero-mode bound.  Trace commutation also
gives $b(R)=0$.
Applying the two estimates to $u_j-u$ proves convergence in the weighted
direct-sum seminorm and in $\Dcal_R$; the triangle inequality and the
standard polarization of the bounded form $Q_0$ give the two asserted
limits.
\end{proof}

We now have all estimates needed to remove the smoothness assumption.  A
general admissible map can be approximated without changing its boundary
trace.  The energies, the quartic term, and every Fourier form pass to the
limit.  Thus the exact defect identity and its nonnegative sign hold on the
full Sobolev class.
\begin{proposition}
\label{6-23}
For every $u\in\Acal_R$, let $\sigma=\frac{F}{f}$ and define
$c_m,a_m,b$ by
\eqref{6-14}, \eqref{6-21}, and
\eqref{6-22}.  Then
\begin{equation} \label{6-24}
\begin{aligned}
 D_R[u]
 :=&\bigl(\mathcal E_R(u)-\mathcal E_R(fe^{i\theta})\bigr)\\
 &-\bigl(\EGL(\sigma u;B_R(0))-\EGL(V;B_R(0))\bigr)\\
 =&Q_0[b]+\sum_{m\neq0}Q_m[a_m]\\
 &+\frac{1}{4}\int_{B_R(0)}
 \left(1-\left(\frac{F}{f}\right)^4\right)
 (|u|^2-f^2)^2\dd x\geq0.
\end{aligned}    
\end{equation}

\end{proposition}

\begin{proof}
Let us write
\[
 u=fe^{i\theta}+v,\qquad v\in H_0^1(B_R(0);\C).
\]
Choose $v_j\in C_c^\infty(B_R(0);\C)$ with $v_j\to v$ in $H^1$.  The formula
$f(r)=rA_f(r^2)$ gives
\[
 f(r)e^{i\theta}=A_f(|x|^2)(x_1+ix_2),
\]
which is smooth at the origin.  Thus
$u_j:=fe^{i\theta}+v_j\in\Acal_R^\infty$ and $u_j\to u$ in $H^1$.
Lemma~\ref{2-6} also gives
$\sigma u_j\to\sigma u$ in $H^1$.

The standard embedding $H^1(B_R(0))\hookrightarrow L^4(B_R(0))$,
together with boundedness of multiplication by $\sigma$ on $H^1$, gives
convergence of both energies and of the quartic remainder.

Let $c_{m,j},a_{m,j},b_j$ be the coefficients associated with $u_j$.
Because $\Tr u_j=\Tr u=e^{i\theta}$, trace compatibility gives
$P_0(\Tr u_j)=P_0(\Tr u)=0$.  Lemma~\ref{6-20}
therefore applies to the entire sequence and yields
\[
 Q_0[b_j]\to Q_0[b],\qquad
 \sum_{m\ne0}Q_m[a_{m,j}]\to\sum_{m\ne0}Q_m[a_m].
\]
Together with the energy and quartic convergences above, these limits
allow passage in the single identity of Proposition~\ref{3-13}
for $u_j$, giving \eqref{6-24}.  The series on
the right converges and has finite sum by Lemma~\ref{6-20}; each
summand is nonnegative, while $Q_0[b]\geq0$ by
Proposition~\ref{6-11} and $1-\left(\frac{F}{f}\right)^4>0$ by
Proposition~\ref{2-4}.   
\end{proof}

\section{Global minimality and uniqueness}
\label{7-1}

In this section, we will conclude the proof of our main result. The proof now reduces to adding two nonnegative quantities.  One is the
whole-plane energy difference, and the other is the disk defect just computed.
For equality, the defect formula determines every Fourier coefficient. The next
proposition gives the scaled theorem and its equality case. 
Our main theorem then readily follows from Proposition \ref{7-2}.
\begin{proposition}
\label{7-2}
For every $R>0$ and every $u\in\Acal_R$,
\begin{equation}\label{7-3}
 \mathcal E_R(u)\geq\mathcal E_R(fe^{i\theta}).
\end{equation}
Moreover,\[
 \mathcal E_R(u)=\mathcal E_R(fe^{i\theta})
 \quad\Longleftrightarrow\quad
 u=fe^{i\theta}\quad\text{almost everywhere in }B_R(0).\]
\end{proposition}

\begin{proof}
Set $\sigma=\frac{F}{f}$.  Lemma~\ref{3-3} and
Proposition~\ref{6-23} give
\[
 \mathcal E_R(u)-\mathcal E_R(fe^{i\theta})
 =D_R[u]+\EGL(\sigma u;B_R(0))-\EGL(V;B_R(0))\geq0,
\]
which proves \eqref{7-3}.

Now let us suppose equality holds. Then both nonnegative terms on the right vanish. In
particular, $D_R[u]=0$. 
In \eqref{6-24}, every summand is finite and nonnegative, so a vanishing sum
forces $Q_0[b]=0$ and
$Q_m[a_m]=0$ for every $m\neq0$.  Note that we can also deduce $|u|=f$, but this alone does not determine the phase and hence is not sufficient for uniqueness.

Proposition~\ref{6-11}
then gives $b=0$, and hence $c_0=\frac{f}{r}b=0$.  If $|m|\geq2$, the strictly
positive zeroth-order weight in
\[
 Q_m[a_m]=\pi\int_0^R
 \left(rd|a_m'|^2+(m^2-1)\frac{d}{r}|a_m|^2\right)\dd r
\]
gives $a_m=c_m=0$.  For $m=\pm1$, the same identity gives
$a_m'=0$ locally on $(0,R)$, so $a_m$ is constant.  On every outer annulus
$a_m$ has an $H^1$ trace at $R$; the boundary identification in
Lemma~\ref{6-13}, together with $f(R)=1$, gives
$a_1(R)=1$ and $a_{-1}(R)=0$.  Thus $c_1=f$ and $c_m=0$ for $m\neq1$.
Completeness of the angular Fourier basis yields
$u=fe^{i\theta}$ almost everywhere.  The converse is immediate and the proof is thus completed.
\end{proof}

\medskip
\noindent\textbf{Acknowledgements.}
Hong-Ge Chen was supported by the National Natural Science Foundation of China (Grant Nos.~12201607 and~12571249) and the Postdoctoral Project of Hubei Province (Grant No.~2024HBBHXF095). Yong Liu was supported by  National Natural Science Foundation of China No. 12471204. Juncheng Wei was supported by National R\&D Program of China (Grant No. 2022YFA1005602), and Hong Kong General Research Fund ``New frontiers in singular limits of nonlinear partial differential equation". Wen Yang was supported by the National Key Research and Development Program of China (Grant No.~2022YFA1006800), the National Natural Science Foundation of China
(Grant No. 12531010), the Science and Technology Development Fund of the Macao SAR (FDCT, Grant No.~0070/2024/\allowbreak RIA1), the Multi-Year Research Grants of the University
of Macau (Grant No. MYRG-GRG2025-00051-FST), and the University of Macau Development Foundation (Grant No.~TISF/2025/006/FST). The authors acknowledge the use of AI tools. All mathematical arguments and proofs in the final manuscript were checked and written by the authors.
\bigskip

\noindent {\bf Data availability statement}: There are no data associated with this article.

\end{document}